\documentclass[11pt]{article}

\usepackage{amsmath}
\usepackage{amssymb}
\usepackage{amsfonts}
\usepackage[unicode, pdftex, final]{hyperref}
\hypersetup{
	colorlinks = true,
	linkcolor=magenta,
	citecolor = red
}
\usepackage{amsthm}

\newtheorem{theorem}{Theorem}[section]
\newtheorem{proposition}[theorem]{Proposition}
\newtheorem{lemma}[theorem]{Lemma}
\newtheorem{corollary}[theorem]{Corollary}
\theoremstyle{remark}
\newtheorem*{remark}{Remark}
\newtheorem*{notation_remark}{Notation remark}
\theoremstyle{definition}

\usepackage{mathtools}
\numberwithin{equation}{section}

\newcommand{\Gam}[2]{\Gamma\biggl[\genfrac{}{}{0pt}{}{#1}{#2}\biggr]}
\newcommand{\FO}[3]{{}_1F_1\biggl[\genfrac{}{}{0pt}{}{#1}{#2}\biggm|#3\biggr]}
\newcommand{\FT}[4]{{}_2F_1\biggl[\genfrac{}{}{0pt}{}{#1, #2}{#3}\biggm|#4\biggr]}
\newcommand{\R}{\mathbb{R}}

\newcommand{\Z}{\mathbb{Z}}
\newcommand{\C}{\mathbb{C}}
\newcommand{\T}{\mathbb{T}}
\newcommand{\I}{\mathbb{I}}
\newcommand{\F}{\mathcal{F}}

\newcommand{\CT}{\mathcal{T}_s}
\newcommand{\CO}[1]{G_{#1}}

\let\oldpsi\psi
\renewcommand{\psi}{\oldpsi_s}
\let\oldrho\rho
\renewcommand{\rho}{\oldrho_s}

\DeclarePairedDelimiter\abs{\lvert}{\rvert}
\DeclarePairedDelimiter\norm{\lVert}{\rVert}
\DeclarePairedDelimiter\bra{(}{)}

\DeclareMathOperator{\sign}{sgn}
\DeclareMathOperator{\Tr}{Tr}

\DeclareMathOperator{\supp}{supp}

\DeclareMathOperator{\wlim}{w-lim\ }
\DeclareMathOperator{\spann}{span}

\newcommand\blfootnote[1]{%
	\begingroup
	\renewcommand\thefootnote{}\footnote{#1}%
	\addtocounter{footnote}{-1}%
	\endgroup
}

\title{Unitary transform diagonalizing the confluent hypergeometric kernel}
\author{Sergei M. Gorbunov${}^*$
\blfootnote{Institute for System Programming of the Russian Academy of Sciences, Moscow, Russia}
\blfootnote{Moscow Institute of Physics and Technology, Dolgoprudny, Moscow Region, Russia}
\blfootnote{Steklov Mathematical Institute of Russian Academy of Sciences, Moscow, Russia}}
\date{}

\begin{document}
\maketitle
\begin{abstract}
We consider the image of the operator inducing the determinantal point process
with the confluent hypergeometric kernel.
The space is described as the image of $L_2[0, 1]$ under a unitary transform,
which generalizes the Fourier transform.
For the derived transform we prove a counterpart of the Paley-Wiener theorem.
We use the theorem to prove that the corresponding analogue of the Wiener-Hopf
operator is a unitary equivalent of the usual Wiener-Hopf operator, which implies
that it shares the same factorization properties and Widom's trace formula.
Finally, using the introduced transform we give explicit formulae for the
hierarchical decomposition of the image of the operator induced by the
confluent hypergeometric kernel.\end{abstract}

\section{Introduction}
\let\thefootnote\relax\footnotetext{This work was supported by
Ministry of Science and Higher Education of the Russian Federation (Grant No. 075-15-2024-529)}

Fix a complex number $s$ such that $\Re s >-1/2$.
For $x, y\in\R$ such that $x\ne y$ consider the following kernel
\begin{equation}\label{1_eq:CHK_def}
        K^s(x, y) = \rho(x)\rho(y)
        \frac{Z_s(x)\overline{Z_s(y)}-e^{i(x-y)}\overline{Z_s(x)}Z_s(y)}{2\pi i(y-x)},
        \end{equation}
        where
            \begin{align*}
    &\Gam{a, b,\ldots}{c, d\ldots} =
    \frac{\Gamma(a)\Gamma(b)\ldots}{\Gamma(c)\Gamma(d)\ldots},
    \quad \rho(x)=\abs{x}^{\Re s}e^{-\frac{\pi}{2}\Im s\sign x},\\
    &Z_s(x) = \Gam{1+s}{1+2\Re s}\FO{\bar{s}}{1+2\Re s}{ix},
    \end{align*}
    and ${}_1F_1$ stands for the confluent hypergeometric function
    defined by the formula~\eqref{A_eq:def}. For $x=y$ define $K^s(x, x)$
    by the L'H\^opital rule. The kernel induces a locally trace-class operator
    of orthogonal projection on $L_2(\R)$ (see Theorem~\ref{2:Diagonalization}
    or~\cite[Corollary~1]{B_23}) and by the Macchi-Soshnikov-Shirai-Takahashi Theorem~\cite{M_75, S_00, ST_03}
    induces a determinantal point process. The process was first derived by Borodin and
    Olshanski as the scaling limit of the Pseudo-Jacobi orthogonal polynomial ensemble
    on the real line~\cite{BO_01}. It may also be derived as the scaling limit of the Jacobi
    circular orthogonal polynomial ensemble; these calculations were done by Bourgade,
    Nikeghbali and Rouault (see~\cite{BNR_06} or Theorem~\ref{5:BNR_th}). For $s=0$ the
    kernel $K^0(x, y)$ reduces to the sine kernel. The sine kernel 
    is a projector onto the Paley-Wiener space of entire functions.
    The purpose of the present paper is to find a similar 
    description of the kernel $K^s(x, y)$ for arbitrary parameter $s$.
    
    The image of the operator induced by the kernel $K^s(x, y)$ was described by Bufetov~\cite{B_23}.
    He decomposed that space into the sum of one-dimensional subspaces. 
    Each subspace is spanned by a function
    defined by its behavior at zero (see Subsection~\ref{1_SS:PW_space}).
    In the present paper we give another description of that space and reproduce Bufetov's result.
    
    Introduce the "generalized exponent"
     \begin{equation}
        \CT(x) = \frac{e^{-ix}}{\sqrt{2\pi}}\rho(x)\overline{\psi(x)}Z_s(x),
    \end{equation}
    where
    \[
    \psi(x)=e^{-\frac{i\pi}{2}\Re s\sign x}\abs{x}^{-i\Im s}.
    \]
    The integral
     \begin{equation}\label{2_eq:CT_def}
   \CT f(\omega) = \int_\R \CT(\omega x)f(x)dx
    \end{equation}
    is a well-defined function of $\omega$, which is continuous on $\mathbb{R}\setminus \{0\}$
    for any $f\in L_1(\R)\cap L_\infty(\R)$. 
    
    \begin{notation_remark}
    Here and subsequently for a kernel $K(x, y)$ we denote the respective operator
    \[
    (Kf)(x) = \int_{-\infty}^\infty K(x, y)f(y)dy
    \]
    by $K$.
    Further, for a function $h\in L_\infty(\R)$ let $h$ also stand for the respective operator of
    pointwise multiplication on $L_2(\R)$:
    \[
    (hf)(x) = h(x)f(x).
    \]
    For a subset $A\subset\R$ by $\I_A$ we denote
    the indicator function of $A$. Let $\I_\pm = \I_{\R_\pm}$. We adopt the following convention
    for the Fourier transform
    \[
    \hat{f}(\omega)= \frac{1}{\sqrt{2\pi}}\F f(\omega)=\frac{1}{2\pi}\int_{-\infty}^\infty e^{-i\omega x}f(x)dx.
    \]
    \end{notation_remark}
    
     \begin{theorem}\label{2:Diagonalization}
        The operator $\CT$ defines an isometry on the dense subset
        $L_1(\R)\cap L_\infty(\R)$ of $L_2(\R)$ and extends to a unitary operator
        diagonalizing the confluent hypergeometric kernel
        \[
         \CT^*\I_{[0, 1]}\CT = \psi K^s\psi^*.
        \]
        The equality may be treated as a relation between the corresponding kernels
        \begin{equation}\label{2_eq:CD_cont}
        \int_0^1\overline{\CT(xt)}\CT(yt)dt = \psi(x)K^s(x, y)\overline{\psi(y)}.
        \end{equation}
    \end{theorem}
    \begin{remark}
    For $s=0$ we have $\F=\mathcal{T}_0$.
    \end{remark}
    \begin{remark}
    Recall that a determinantal measure is invariant under gauge transformations
    of the corresponding operator. Gauge transformations are conjugations of an
    operator $K$ by an operator of pointwise multiplication by a function $\varphi K\varphi^*$ for a function
    $\varphi$ satisfying $\abs{\varphi}=1$.
    \end{remark}
    
    We prove the relation~\eqref{2_eq:CD_cont} in Subsection~\ref{sec:christ_darb}. We explain how that relation implies Theorem~\ref{2:Diagonalization} in Subsection \ref{sec:main_th_proof}.
    
    \subsection{The Paley-Wiener space and the Paley-Wiener Theorem for $\CT$}\label{1_SS:PW_space}
    Let $\mathcal{PW}_s$ stand for the image of the operator $K^s$. For $s=0$ the
    operator $\CT$ coincides with the Fourier transform; the corresponding
    space $\mathcal{PW}_0$ is the Paley-Wiener space --- the space of entire
    functions, whose Fourier transforms are supported on $[0, 1]$. A description of
    $\mathcal{PW}_s$ was given by Bufetov in~\cite{B_23}. It is shown that any
    function in $\mathcal{PW}_s$ extends to an entire function multiplied by
    $\rho(x)$. To be precise, introduce the subspaces
    \[
    H^{(s, n)} = \left\{f\in\mathcal{PW}_s: f(x) = \rho(x)h_f(x), h_f(z) = O(z^n), z\to 0\right\},
    \]
    where $h_f(z)$ is an entire function. Observe that $H^{(s, n+1)}\subset H^{(s, n)}$.
    Let $L^{(s, n)}$ be the orthogonal complement of $H^{(s, n+1)}$ in $H^{(s, n)}$. We have
     \[
    \mathcal{PW}_s= \bigoplus_{n\in\Z_{\ge 0}}L^{(s, n)}.
    \]
    In~\cite{B_23} it is further proved that $L^{(s, n)}$ are one-dimensional. 
    Using Theorem \ref{2:Diagonalization} we are able to reproduce that. Furthermore, we
    give explicit formulae for the functions spanning $L^{(s, n)}$.
    
    \begin{corollary}\label{2:H_decomposition}
    \begin{itemize}
    \item The space $\mathcal{PW}_s$ is the image of $L_2[0, 1]$ under $\psi^*\CT^*$.
    \item Any $f\in \mathcal{PW}_s$ is an entire function multiplied by $\rho(x)$.
    \item Introduce the subspaces
    \[
    F^{(s, n)} = \I_{[0, 1]}t^{s}\spann\langle 1, t, \ldots, t^{n-1}\rangle\subset L_2[0, 1].
    \]
    We have that $H^{(s, n)}$ is the image of $\bra*{F^{(s, n)}}^{\perp}\cap L_2[0, 1]$ under $\psi^*\CT^*$.
    \item The functions spanning $L^{(s, n)}$ may be expressed in terms of orthogonal polynomials.
    Denote by $\left\{P_n^{(2\Re s)}\right\}_{n\ge 0}$ orthogonal polynomials with respect to
    the weight $t^{2\Re s}$ on $[0, 1]$. Then $L^{(s, n)}$ is
    spanned by 
    $$
    \mathcal{L}_{(s, n)} = \psi^*\CT^*\bra*{\I_{[0, 1]}(t)t^{s}P^{(2\Re s)}_n(t)}.
    $$
    \end{itemize}
    \end{corollary}
    We deduce Corollary \ref{2:H_decomposition} from Theorem \ref{2:Diagonalization} in Section \ref{5_sec:H_dec}.
    
    Observe that our $P_n^{(2\Re s)}(t)$ are essentially the Jacobi orthogonal polynomials $P^{(2\Re s, 0)}_n(1-2t)$(see~\cite[22.2.1]{AS_64}).
    They may be explicitly expressed in terms of the Hypergeometric function ${}_2F_1$
    defined by the formula~\eqref{A_eq:def}. We have (see~\cite[22.5.42]{AS_64}) that up
    to a constant factor
    \[
    P_n^{(2\Re s)}(t) = \FT{-n}{n+1+2\Re s}{1+2\Re s}{t}.
    \]
    We conclude that $L^{(s, n)}$ is spanned by
    \[
    \mathcal{L}_{(s, n)}(x) = \abs{x}^{\Re s}e^{-\frac{\pi}{2}\Im s\sign x}\int_0^1e^{ixt}
    \FO{s}{1+2\Re s}{-ixt}t^{2\Re s} \FT{-n}{n+1+2\Re s}{1+2\Re s}{t}dt.
    \]
    
    Recall that the Paley-Wiener Theorem (see Theorem~\ref{5:PW_theorem}, \cite[Theorem~19.2]{R_86})
    asserts that the Hardy space $H^2(\mathbb{H})$ of functions, which extend analytically to the
    upper half-plane $\mathbb{H}$, coincides with the space $\F^*L_2(\R_+)$ of functions, whose Fourier transform is supported on $\R_+$. We are able to prove that the Hardy space also coincides with the space $\CT^*L_2(\R_+)$ of functions, whose transform $\CT$ is supported on $\R_+$,
    for general $s$.
    \begin{theorem}\label{2:PW_spaces}
    We have
    \[
    \CT^*\I_+\CT = \F^*\I_+\F.
    \]
    \end{theorem}
    See Subsection \ref{sec:main_th_proof} for the proof of Theorem \ref{2:PW_spaces}.

    \subsection{Wiener-Hopf factorization}
    
    For $f\in L_\infty(\R)$ the Wiener-Hopf operator is defined by the formula
    \[
    W_f = \I_+\F f\F^*\I_+.
    \]
    Similarly, using the introduced transform, define
    \[
    \CO{f} = \I_+\CT f\CT^*\I_+.
    \]
    For $s=0$ we have $\CO{f}=W_f$. Let $\F^* L_1(\R)$ stand for the image of $L_1(\R)$
    under the inverse Fourier transform. This space is an algebra with pointwise multiplication.
    It may be decomposed into subalgebras $\F^*L_1(\R)\simeq \F^* L_1(\R_+)\oplus \F^*L_1(\R_-)$,
    corresponding to functions, whose Fourier transform is supported on $\R_\pm$.
    One can show that the Wiener-Hopf operator preserves multiplication on these subalgebras:
    \[
    W_{fg} = W_fW_g, \quad f, g\in\F^*L_1(\R_\pm).
    \]
    The factorization has been used by Widom~\cite{W_82} to derive the trace formula for the
    Wiener-Hopf operators, which, we note, implies the Central Limit Theorem for the sine process.
    
    Recall that the $1/2$-Sobolev space $H_{1/2}(\R)$ is a Hilbert space
    of functions with the following norm
    \[
    \norm{f}_{H_{1/2}}^2 = \norm{f}_{L_2}^2+\norm{f}_{\dot{H}_{1/2}}^2,
    \quad \norm{f}_{\dot{H}_{1/2}}^2 =
    \int_\R\abs{\omega}\abs{\hat{f}(\omega)}^2d\omega.
    \]
    For $f\in H_{1/2}(\R)\cap\F^*L_1(\R)$ denote its decomposition
    into the positive and negative frequencies $f=f_++f_-$, $\supp \hat{f}_\pm\subset \R_\pm$.
    We have that $[W_{f_-}, W_{f_+}]$ is trace-class, and its trace is equal to
    \[
    \Tr [W_{f_-}, W_{f_+}] = \int_0^\infty \omega \hat{f}(\omega)\hat{f}(-\omega)d\omega.
    \]
    See, for example,~\cite[Sect.~5.2]{BE_03} for these statements.
    For completeness we include the proof to Section~\ref{6_sec:WH_fact}.
    
    By Theorem \ref{2:PW_spaces} we have that $W_f$ and $G_f$ are unitarily equivalent
    \[
    G_f = \CT\F^*W_f\F\CT^*.
    \]
    We conclude with the following corollary of Theorem \ref{2:PW_spaces}.
    \begin{corollary}\label{2:WH_factorization}
    \begin{itemize}
    \item For $f, g\in \F^*L_1(\R_\pm)$ we have that $G_{fg} = G_fG_g$.
    For $f_\pm\in\F^*L_1(\R_\pm)$ the equality $G_{f_+f_-} = G_{f_-}G_{f_+}$ holds.
    \item For $f\in H_{1/2}(\R)\cap\F^*L_1(\R)$ the commutator
    $[G_{f_-}, G_{f_+}]$ is trace-class. Its trace is given by
    \[
    \Tr[G_{f_-}, G_{f_+}] = \int_0^\infty \omega \hat{f}(\omega)\hat{f}(-\omega)d\omega.
    \]
    \end{itemize}
    \end{corollary}
    Corollary \ref{2:WH_factorization} is proved in Section \ref{6_sec:WH_fact}.
    
    \subsection{Related work}
    As mentioned in the beginning, by the Macchi-Soshnikov-Shirai-Takahashi Theorem~\cite{M_75, S_00, ST_03} the
    kernel $K^s$ induces a determinantal point process $\mathbb{P}_{K^s}$.
    Apart from the constructions of the process given in~\cite{BNR_06, BO_01},
    we recall several more.
    
    The filtration $H^{(s, n)}$ of the spaces $\mathcal{PW}_s$ (see subsection \ref{1_SS:PW_space})
     may be interpreted in terms
    of the Palm hierarchy. In~\cite{B_23} it is shown that the Palm measure of $\mathbb{P}_{K^s}$
    at zero is $\mathbb{P}_{K^{s+1}}$. Therefore if the parameter $s$ is a positive integer,
    the process is the $s$-th Palm measure of the sine process at zero.
    Recall that by the Shirai-Takahashi Theorem~\cite{ST_03P}
    the image of the operator corresponding to the Palm measure differs by a one-dimensional subspace.
    Though the theorem is not applicable directly to the kernel $K^s$,
    the assertion still holds. In~\cite{B_23} it is shown that $H^{(s+1, n)} = \phi H^{(s, n+1)}$
    for some function $\phi$, $\abs{\phi}=1$. In particular, $\mathcal{PW}_{s+1}$ is the image
    of the orthogonal complement of $L^{(s, 0)}$ in $\mathcal{PW}_s$ under the
    multiplication operator $\phi$.
    
    Another construction of $\mathbb{P}_{K^s}$ is the degeneration
    of the more general ${}_2F_1$ determinantal point process~\cite{BD_01}
    under a certain scaling limit.
    
    Let us also recall an interesting connection between the point process
    $\mathbb{P}_{K^s}$ and the space $\mathcal{PW}_s$.
    The Lyons-Peres conjecture, which has been proved by Bufetov, Qiu and Shamov~\cite{BQS_21},
    states that a discrete subset of $\mathbb{R}$ is $\mathbb{P}_K$-almost surely
    a completeness set for a reproducing kernel Hilbert space with the kernel $K(x, y)$.
    This result and Theorem \ref{2:Diagonalization} immediately imply that the functions
    $\{\CT(x_j\cdot)\}_{x_j\in X}$ are dense in $L_2[0, 1]$ for
    $\mathbb{P}_{K^s}$-almost every discrete subset $X\subset\mathbb{R}$.
    
    To make parallels with other processes, we recall the Bessel and the Airy kernel
    determinantal point processes~\cite{TW_93, TW_94}. One general feature of these
    processes is the integrable form of the kernel. Such form yields a connection
    between gap asymptotics and Painlev\'e equations (see~\cite{DKV_11, BD_01}
    for these calculations for $\mathbb{P}_{K^s}$). However, we note that the
    existence of an explicit formula for the diagonalizing unitary transform was used by
    Basor, Ehrhardt and Widom~\cite{BE_03, BW_99} to derive the convergence of additive
    functionals to the Gaussian distribution for the mentioned processes.
    The transform $\CT$ is a counterpart of the Airy transform and the Hankel transform
    diagonalizing the Airy kernel and the Bessel kernel respectively (see~\cite{BE_03, BW_99} for details).
    
    \subsection{Acknowledgements}
    I am deeply grateful to Alexander Bufetov and Alexei Klimenko for useful discussions. I am grateful to the anonymous referees for valuable remarks and comments.
    
    \section{Outline of proof}
    \subsection{Scaling limit of the Christoffel-Darboux formula: proof of the relation \eqref{2_eq:CD_cont}}\label{sec:christ_darb}
    Define the function on the unit circle $\T=\{z\in\mathbb{C}: \abs{z}=1\}$
    \[
    w_s\bra*{z} = \frac{1}{2\pi}\Gam{1+s, 1+\bar{s}}{1+2\Re s}\bra*{1-z}^{\bar{s}}\bra*{1-\bar z}^s, \quad z\in\T.
\]
Let $\{\varphi_n\}_{n\in \Z_{\ge 0}}$ be the orthonormal polynomials 
$$\varphi_n(z)=\kappa_n z^n + (\text{lower degree terms}), \quad\kappa_n>0$$
with respect to the weight $w_s(z)d\theta$, $z=e^{i\theta}$. An exact formula for them is given in Theorem~\ref{A:orth_formula}. Cut the unit circle at $-1$: $\T\setminus \{-1\} = \{e^{i\theta}, \theta\in (-\pi, \pi)\}$. Recall that the Christoffel-Darboux formula \cite[Theorem~2.2.7]{S_05OPUC} states
\begin{multline}\label{3_eq:CD_formula}
K_n\bra*{e^{i\tau}, e^{i\theta}} =  \sqrt{w_s\bra*{e^{i\theta}}w_s\bra*{e^{i\tau}}}\sum_{j=0}^{n-1}\varphi_j\bra*{e^{i\tau}}\overline{\varphi_j\bra*{e^{i\theta}}} =\\
= \sqrt{w_s\bra*{e^{i\theta}}w_s\bra*{e^{i\tau}}}\frac{\overline{\varphi_n^*(e^{i\theta})}\varphi^*_n(e^{i\tau})-\overline{\varphi_n(e^{i\theta})}\varphi_n(e^{i\tau})}{1-e^{i(\tau-\theta)}},
\end{multline}
where $\varphi^*_j(z) = z^j\overline{\varphi_j(1/\bar{z})}$ are the reversed polynomials.
\begin{theorem}[Bourgade, Nikeghbali, Rouault {\cite[Theorem~5]{BNR_06}}]\label{5:BNR_th}
We have
\[
\frac{1}{n}K_n\bra*{e^{ix/n}, e^{iy/n}}\to K^s(x, y), \text{ as }n\to\infty.
\]
\end{theorem}
\begin{remark}
The kernel derived in \cite{BNR_06} differs from $K^s$ defined by the formula \eqref{1_eq:CHK_def} by a conjugation by $e^{ix/2}$ and interchanging $x$ and $y$. As was already mentioned, the induced point process does not change. However, in order to derive the formula \eqref{2_eq:CD_cont} it will be important that $K^s$ is the limit of $\frac{1}{n}K_n\bra*{e^{ix/n}, e^{iy/n}}$.
\end{remark}

The theorem is proven by direct taking limit of the right-hand side of the Christoffel-Darboux formula \eqref{3_eq:CD_formula}. For a positive $c$ let $[c]$ be its integer part. To derive the formula \eqref{2_eq:CD_cont} express the left-hand side of the identity \eqref{3_eq:CD_formula} as follows
\begin{equation}\label{5_eq:int_repr}
\frac{1}{n}K_n\bra*{e^{i\theta/n}, e^{i\tau/n}} =  \sqrt{w_s\bra*{e^{i\theta/n}}w_s\bra*{e^{i\tau/n}}}\int_0^1\varphi_{[nt]}\bra*{e^{i\theta/n}}\overline{\varphi_{[nt]}\bra*{e^{i\tau/n}}}dt.
\end{equation}
The relation \eqref{2_eq:CD_cont} follows from the convergence of $\varphi_{[nt]}\bra*{e^{i\tau/n}}$.
\begin{lemma}\label{5:ker_conv}
The following convergence
\[
\overline{\CT^n(y, x)} = [nx]^{-i\Im s}\sqrt{w_s\bra*{e^{iy/n}}}\varphi_{[nx]}\bra*{e^{iy/n}} \to \overline{\CT(xy)\psi(xy)}, \quad \text{ as }n\to\infty
\]
takes place locally uniformly for $(x, y)\in (0, \infty)\times \R$.
Further, the estimate
\[
\abs*{\CT^n(y, x)}\le C\abs{y}^{\Re s}\bra*{1+\abs{x}^{\Re s}}
\]
holds for some independent of $n$ constant $C$. That estimate is locally uniform for $(x, y)\in [0, \infty)\times \R$.
\end{lemma}
The proof of Lemma \ref{5:ker_conv} is completely parallel to the proof of Theorem \ref{5:BNR_th} in \cite{BNR_06}. We present this in Section \ref{3_sec:CD_limit}.
\begin{proof}[Proof of the identity \eqref{2_eq:CD_cont}]
A direct substitution of the asymptotics from Lemma \ref{5:ker_conv} into the right hand side of formula \eqref{5_eq:int_repr} gives as $n\to\infty$
\begin{multline*}
\sqrt{w_s\bra*{e^{ix/n}}w_s\bra*{e^{iy/n}}}\int_0^1\varphi_{[nt]}\bra*{e^{ix/n}}\overline{\varphi_{[nt]}\bra*{e^{iy/n}}}dt=\\
= \int_0^1\overline{\CT^n(x, t)}\CT^n(t, y)dt\to \int_0^1\overline{\psi(xt)\CT(xt)}\CT(yt)\psi(yt)dt,
\end{multline*}
where the convergence of the integral follows from the dominated convergence theorem and the estimate in Lemma \ref{5:ker_conv}. We note that for $t>0$ we have $\psi(xt)\overline{\psi(yt)} = \psi(x)\overline{\psi(y)}$. Hence the right-hand side of the equality \eqref{5_eq:int_repr} converges to the left-hand side of the relation \eqref{2_eq:CD_cont} multiplied by $\psi(x)\overline{\psi(y)}$. The convergence of the left-hand side of the equality \eqref{5_eq:int_repr} follows from Theorem~\ref{5:BNR_th}.
\end{proof}

    \subsection{Proof of Theorems \ref{2:Diagonalization} and \ref{2:PW_spaces}
    from the relation \eqref{2_eq:CD_cont}}\label{sec:main_th_proof}
    We show that $\CT$ is unitary using the following criterion for the multiplication operators.
    \begin{proposition}\label{5:diag_mult}
    A bounded linear operator $J$ on $L_2(\R)$ is an operator of pointwise multiplication if and only if for any Borel disjoint sets $A, B\subset \R$ we have $\I_A J\I_B=0$, where $\I_A$ and $\I_B$ are operators of pointwise multiplication by the indicator functions of $A$ and $B$.
    \end{proposition}
    See Section \ref{4_sec:unit} for the proof of Proposition \ref{5:diag_mult}
    
    Using the boundedness of $\CT^*$ (see Lemma \ref{4:boundedness}) and the identity \eqref{2_eq:CD_cont}, we deduce that Theorem \ref{2:PW_spaces} holds after restriction to disjoint Borel sets.
    \begin{lemma}\label{5:T_diagonal}
    For any $\varepsilon > 0$ and any compact Borel disjoint sets $A, B\subset \R\setminus [-\varepsilon, \varepsilon]$ satisfying $\abs{x-y}>\varepsilon$ for any $x\in A$, $y\in B$ the equality
    $$
    \I_A(\CT^*\I_\pm\CT-\F^*\I_\pm\F)\I_B=0
    $$
    holds.
    \end{lemma}
    Lemma \ref{5:T_diagonal} is proved in Section \ref{4_sec:unit}.
    
    The assertion of Lemma \ref{5:T_diagonal} may be extended to arbitrary disjoint Borel sets $A, B\subset \R$ by continuity. By Proposition \ref{5:diag_mult} this implies that $\CT^*\CT=g$ for some $g\in L_\infty(\R)$. Observe, that the operator $\CT^*\CT$ is invariant under conjugation by the dilation operator $D_Rh(x)=h(x/R)$ for any $R\ne 0$. Thereby we conclude that $g=C\in\R$. The same holds for $\CT\CT^*$ since $\CT^*= \mathcal{J}\mathcal{T}_{\bar{s}}$, where $\mathcal{J}f(x)=f(-x)$. To show that $C=1$ it is sufficient to establish
    \begin{lemma}\label{5:T_norm}
    We have that $\norm{\CT\I_{[n, n+1]}}_{L_2}\to 1$ as $n\to \infty$.
    \end{lemma}
    We prove Lemma \ref{5:T_norm} in Section \ref{4_sec:unit}.
    
    Let us explain how Lemma \ref{5:T_norm} yields that $C=1$. The identity $\CT^*\CT=C$ implies
    \[
    \norm*{\CT \I_{[n, n+1]}}_{L_2}^2 = \left\langle \CT  \I_{[n, n+1]}, \CT  \I_{[n, n+1]}\right\rangle_{L_2} =  \left\langle  \I_{[n, n+1]}, \CT^*\CT  \I_{[n, n+1]}\right\rangle_{L_2} = C.
    \]
    The constant $C$ does not depend on $n$ and converges to $1$ as $n\to\infty$. Hence $C=1$ identically.
    This finishes the proof of Theorem \ref{2:Diagonalization}.
    
    Theorem \ref{2:PW_spaces} similarly follows from Lemma \ref{5:T_diagonal}. Applying again Proposition \ref{5:diag_mult}, we have that for some $u_\pm\in L_\infty(\R)$ the difference is the operator of pointwise multiplication
    \[
    \CT^*\I_\pm\CT - \F^*\I_\pm\F = u_\pm.
    \]
    The above difference is invariant under the conjugation by $D_R$ for $R>0$. Applying conjugation by $D_{-1}$, we deduce that $u_+(x)=u_-(-x)$. Theorem \ref{2:Diagonalization} yields that $u_++u_-=0$. Thereby $u_+(x)=C_s\sign x$ for some $C_s$.
    
    Let us show that $C_s=0$. Consider a function $q(x)=(\CT^*\I_{[1/2, 1]})(x)$. By Lemma \ref{5:T_diagonal} we have
    \[
    (\F^*\I_+\F q)(x) = (1-C_s\sign x)q(x).
    \]
    The claim follows from the following statement.
    \begin{lemma}\label{5:T_szero}
    The function $q$ belongs to the space $\F^*L_2(\R_+)$.
    \end{lemma}
    See Section \ref{5_sec:H_dec} for the proof of Lemma \ref{5:T_szero}.
    
    To conclude the proof of Theorem \ref{2:PW_spaces} recall that by the Uniqueness Theorem for the Hardy space (see \cite[Corollary~4.2]{G_07}) and the Paley-Wiener Theorem any subset of positive measure of $\R$ is a uniqueness set for $\F^*L_2(\R_+)$. It follows that
    \[
    C_s(1+\sign x)q(x)\in\F^*L_2(\R_+).
    \]
    The function above is zero on $\R_-$ and is therefore zero identically by the Uniqueness Theorem. Observe, however, that the latter holds only if $C_s=0$ by the unitarity of $\CT^*$ and the first claim of Corollary \ref{2:H_decomposition}.

    \subsection{Structure of the paper}
    The rest of the paper has the following structure. In Section \ref{3_sec:CD_limit} we prove Lemma \ref{5:ker_conv} and conclude that the identity \eqref{2_eq:CD_cont} holds. Using the convergence asserted in the lemma we deduce the boundedness of $\CT^*$, $\CT$ (see Lemma \ref{4:boundedness}). In Section \ref{4_sec:unit} we prove Proposition \ref{5:diag_mult} and Lemmata \ref{5:T_diagonal}, \ref{5:T_norm}, which imply Theorem \ref{2:Diagonalization}. In Section \ref{5_sec:H_dec} we prove Corollary \ref{2:H_decomposition} from Theorem \ref{2:Diagonalization} and deduce Lemma \ref{5:T_szero} from Corollary \ref{2:H_decomposition}, which concludes the proof of Theorem \ref{2:PW_spaces}. Section \ref{6_sec:WH_fact} is devoted to the proof of Corollary \ref{2:WH_factorization} from Theorem \ref{2:PW_spaces}.

    \section{Asymptotic of the Christoffel-Darboux kernel}\label{3_sec:CD_limit}
    In this section we prove Lemma \ref{5:ker_conv} and deduce the boundedness of $\CT$. Let $\{\Phi_n\}_{n\ge 0}$ stand for the monic orthogonal polynomials with respect to the weight $w_s$. See Theorem \ref{A:orth_formula} for explicit formulas. Recall that the Stirling formula \cite[6.1.41]{AS_64} asserts that for $\abs{\arg z}<\pi$ we have as $\abs{z}\to \infty$
    \[
\ln \Gamma(z) = \left(z-\frac{1}{2}\right)\ln z-z+\frac{1}{2}\ln \pi + O(z^{-1}).
\]
For $a, b\in\mathbb{C}$ such that $\Re a>0$ there exists a constant $A$ such that the inequality
\begin{equation}\label{3_eq:stir_form}
\abs*{(1+x)^{b-a}\Gam{a+x}{b+x} -1}\le\frac{A}{1+x}
\end{equation}
holds for all $x\ge 0$.
    \begin{lemma}\label{3:1_conv}
    We have
    \[
    \bra*{1+[nx]}^{-s}\Phi_{[nx]}\bra*{e^{iy/n}} \to e^{ixy}\overline{Z_s(xy)}, \quad\text{ as } n\to\infty
    \]
    locally uniformly for $(x, y)\in (0, \infty)\times \R$.
    The estimate
    \[
    \abs*{\bra*{1+[nx]}^{-s}\Phi_{[nx]}\bra*{e^{iy/n}}} =O(1), \quad\text{ as } n\to\infty
    \]
    holds locally uniformly for $(x, y)\in [0, \infty)\times \R$ .
    \end{lemma}
    \begin{proof}
   By the formula \eqref{3_eq:stir_form} the following bound
    \[
    \bra*{1+[nx]}^{-s}\Gam{1+2\Re s +[nx]}{1+[nx]+\bar{s}} = 1+O\bra*{\frac{1}{1 + [nx]}}, \quad \text{ as }n\to\infty
    \]
    holds uniformly for $x\ge 0$.
   From the integral representation \eqref{A_eq:int_2} we deduce
   \[
   \FT{-[nx]}{1+\bar{s}}{1+2\Re s}{1-e^{iy/n}} = \Gam{1+2\Re s}{1+\bar{s}, s}\int_0^1t^{\bar{s}}(1-t)^{s-1}\bra*{1-t\bra*{1-e^{iy/n}}}^{[nx]}dt,
   \]
   where
    \begin{align*}
  \bra*{1-t\bra*{1-e^{-iy/n}}}^{[nx]} &= \exp\bra*{[nx]\ln (1 + \frac{ity}{n} + O(ty^2/n^2))}\\
  &=\exp\bra*{[nx]\bra*{\frac{ity}{n} + O(ty^2/n^2)+O(t^2y^2/n^2)}}\\
  &= \exp(ityx)(1 + O(txy^2/n)).
  \end{align*}
 Using the integral representation \eqref{A_eq:int} we conclude that locally uniformly on $[0, \infty)\times \R$
\[
  \FT{-[nx]}{\bar{s}+1}{2\Re s + 1}{1-e^{iy/n}} \to \FO{\bar{s}+1}{2\Re s + 1}{ixy}.
\]
A direct substitution into the formula given in Theorem \ref{A:orth_formula} and application of Kummer's formula \eqref{A_eq:kummer} finish the proof of the convergence.
    \end{proof}
\begin{lemma}\label{3:2_conv}
The following convergence
\[
\norm{\Phi_{[nx]}}_{L_2}^{-2}\to \Gam{1+2\Re s}{1+s, 1+\bar{s}},\quad 
n^{\Re s}\sqrt{w_s\bra*{e^{iy/n}}}\to \frac{1}{\sqrt{2\pi}}\sqrt{\Gam{1+s, 1+\bar{s}}{1+2\Re s}}\rho(y), \quad \text{ as }n\to\infty
\]
takes place locally uniformly for $(x, y)\in (0, \infty)\times \R$.
The bounds
\[
\norm{\Phi_{[nx]}}_{L_2}^{-2}=O(1), \quad \abs*{n^{\Re s}\sqrt{w_s\bra*{e^{iy/n}}}}=O\bra*{\abs{y}^{\Re s}}, \quad \text{ as }n\to \infty.
\]
hold locally uniformly for $(x, y)\in [0, \infty)\times \R$.
\end{lemma}
\begin{proof}
The formula \eqref{3_eq:stir_form} yields
\[
\Gam{1+s+[nx], 1+\bar{s}+[nx]}{1+[nx], 1+2\Re s+[nx]} = 1+O\bra*{\frac{1}{1+[nx]}}.
\]
The convergence of the norm follows from Theorem \ref{A:orth_formula}.

For the weight we have
\begin{multline*}
n^{\Re s}(1 - e^{iy/n})^{\bar{s}/2}(1-e^{-iy/n})^{s/2} = n^{\Re s}(-iy/n)^{\bar{s}/2}(iy/n)^{s/2}(1 + O(sy/n))
=\\
= \rho(y)(1 + O(sy/n)).
\end{multline*}
\end{proof}
\begin{proof}[Proof of Lemma \ref{5:ker_conv}]
Substituting the formulae from Lemmata \ref{3:1_conv}, \ref{3:2_conv}, we have the convergence
\[
\frac{n^{\Re s}}{\bra*{1+[nx]}^s}\sqrt{w_s\bra*{e^{iy/n}}}\varphi_{[nx]}\bra*{e^{iy/n}} = \frac{n^{\Re s}}{\bra*{1+[nx]}^s}\sqrt{w_s\bra*{e^{iy/n}}}\frac{\Phi_{[nx]}\bra*{e^{iy/n}}}{\norm{\Phi_{[nx]}}_{L_2}}\to \rho(y)\frac{e^{ixy}}{\sqrt{2\pi}}\overline{Z_s(xy)}
\]
as $n\to\infty$. This convergence is locally uniform for $(x, y)\in (0, \infty)\times \R$.
The asymptotic behaviour
\[
\frac{\bra*{1+[nx]}^s}{n^{\Re s}} \sim [nx]^{i\Im s}\abs{x}^{\Re s}, \text{ as }n\to\infty
\]
takes place locally uniformly for $x\in (0, +\infty)$.
Finally, the bound
\[
\abs*{\frac{\bra*{1+[nx]}^s}{n^{\Re s}}}=O\bra*{1+\abs{x}^{\Re s}}
\]
holds uniformly for $x\in [0, \infty)$.
To conclude the proof it is remaining to observe that for $x>0$ we have
$\abs{x}^{\Re s} \rho(y) = \rho(xy)$.

\end{proof}

Let us show how boundedness of $\CT^*$ follows from Lemma \ref{5:ker_conv}
\begin{lemma}\label{4:boundedness}
The operators $\CT^*$ and $\CT$ extend by continuity to bounded operators.
\end{lemma}
\begin{proof}
It is sufficient to establish the assertion for $\CT^*$. Let $h$ be a Borel bounded function supported on $[\varepsilon, b]\subset \R_+$ for some $\varepsilon >0$. We have
\[
\norm{\CT^*h}^2_{L_2} = \lim_{k\to\infty}\norm{\I_{[-k, k]}\CT^*h}^2_{L_2},
\]
where using Lemma \ref{5:ker_conv} we express the norm as follows
\[
\norm{\I_{[-k, k]}\CT^*h}^2_{L_2} = \lim_{n\to\infty}\norm{\I_{[-k, k]}\CT^{n*}h}^2_{L_2}, \quad (\I_{[-k, k]}\CT^{n*}h)(x) = \I_{[-k, k]}(x)\int_\varepsilon^b\overline{\CT^n(x, y)}h(y)dy.
\]
Observe that the operator $\CT^{n*}$ is a partial isometry, with orthogonal complement to the kernel spanned on the indicator functions $\I_{[l/n, (l+1)/n]}$ for $l\in\Z_{\ge 0}$. On the latter this operator acts by 
\[
\I_{[l/n, (l+1)/n]}(y)\mapsto [l]^{-i\Im s}\frac{1}{n}\sqrt{w_s\bra*{e^{iy/n}}}\varphi_l\bra*{e^{iy/n}}
\]
 and preserves the norm. Thereby we have
\[
\norm{\I_{[-k, k]}\CT^{n*}h}^2_{L_2}\le \norm{h}^2_{L_2}.
\]

We have shown that $\CT^*$ extends to a contraction on $L_2(\R_+)$. Since $\mathcal{J}\CT^*=\CT^*\mathcal{J}$, where $\mathcal{J}f(x) = f(-x)$, we have that $\CT^*$ extends to a contraction $L_2(\R_-)$. We conclude that it extends to a bounded operator on $L_2(\R)$ with the norm of at most $2$.
\end{proof}

\section{Unitarity of $\CT$}\label{4_sec:unit}
In this section we conclude the unitarity of $\CT$ by proving Proposition \ref{5:diag_mult} and Lemmata \ref{5:T_diagonal}, \ref{5:T_norm}.
\begin{proof}[Proof of Proposition \ref{5:diag_mult}]
Let an operator $J$ satisfy the condition of the proposition. Define a Borel function $h$ for any bounded Borel set $B\subset\R$ by
\[
h(x) = (J\I_B)(x), \text{for }x\in B.
\]
Observe that its definition does not depend on the choice of $B$: for bounded Borel sets $B_1, B_2\subset \R$ we have
\[
(J\I_{B_1})(x)=(J\I_{B_2})(x), \text{ for }x\in B_1\cap B_2.
\]

Recall that we denote functions and the respective operators of pointwise multiplication by one symbol.
The assumption of the proposition implies that the equalities of operators
\begin{align*}
&J\I_B = \I_BJ\I_B + \I_{\R\setminus B}J\I_B = \I_BJ\I_B,\\
&\I_B J = \I_B J\I_B + \I_BJ\I_{\R\setminus B} = \I_BJ\I_B
\end{align*}
take place. Therefore, the identity
\[
\I_{B_1}\I_{B_2}J\I_{B_2}= \I_{B_1}\I_{B_2}J\I_{B_1} = \I_{B_1}\I_{B_2}J\I_{B_1\cap B_2}.
\]
holds. We conclude that the function $h(x)\in L_\infty(\R)$ is well-defined. 

The argument above implies that $J$ commutes with multiplications by indicator functions.
Via extension by continuity, the operator $J$ commutes with all pointwise multiplication operators. 
Therefore, for a bounded Borel set $B\subset\R$ and $f\in L_\infty(B)$ we have
\[
(Jf)(x) = (Jf\I_B)(x)=f(x)(J\I_B)(x) = h(x)f(x).
\]
We conclude that $J=h$.
\end{proof}

Before diving into calculations let us establish a convenient asymptotic formula. Recall the notation
\[
\rho(x)=\abs{x}^{\Re s}e^{-\frac{\pi}{2}\Im s\sign x} ,\quad \psi(x)=e^{-\frac{i\pi}{2}\Re s\sign x}\abs{x}^{-i\Im s}.
\]
\begin{lemma}\label{4:hyper_est}
For a certain constant $C$, we have the bound
\[
\abs*{Z_s(x)\rho(x)\psi(x)-1}\le \begin{cases}
\frac{C\abs{x}^{\Re s}}{1+\abs{x}^{1+\Re s}}, \text{ if }\Re s < 0,\\
\frac{C}{1+\abs{x}}, \text{ otherwise}.
\end{cases}
\]
\end{lemma}
\begin{proof}
Indeed, substituting the expansion \eqref{A_eq:asymp}, we have as $x\to +\infty$
\begin{multline*}
Z_s(x) = \Gam{1+s}{1+2\Re s}\FO{\bar{s}}{1+2\Re s}{ix}=e^{i\pi\bar{s}}\abs{x}^{-\bar{s}}e^{-\frac{i\pi}{2}\bar{s}}\bra*{1+O(\abs{x}^{-1})} +\\
+\Gam{1+s}{\bar{s}}e^{ix}\abs{x}^{-1-s}e^{-(1+s)\frac{i\pi}{2}}\bra*{1+O(\abs{x}^{-1})} = \frac{1}{\rho(x)\psi(x)}\bra*{1+O(\abs{x}^{-1})}.
\end{multline*}
As $x\to -\infty$ we similarly have
\begin{multline*}
Z_s(x) = \Gam{1+s}{1+2\Re s}\FO{\bar{s}}{1+2\Re s}{ix}=e^{-i\pi\bar{s}}\abs{x}^{-\bar{s}}e^{\frac{i\pi}{2}\bar{s}}\bra*{1+O(\abs{x}^{-1})} +\\
+\Gam{1+s}{\bar{s}}e^{-ix}\abs{x}^{-1-s}e^{(1+s)\frac{i\pi}{2}}\bra*{1+O(\abs{x}^{-1})} = \frac{1}{\rho(x)\psi(x)}\bra*{1+O(\abs{x}^{-1})}.
\end{multline*}

We conclude that as $\abs{x}\to\infty$ the asymptotic expression
\[
Z_s(x)\psi(x)\rho(x) = 1 + O\bra*{\frac{1}{\abs{x}}}
\]
holds, which implies the claim for large $x$. For small $x$ we have
\[
Z_s(x)\rho(x)\psi(x) - 1 = O(1+\abs{x}^{\Re s}), \text{ as }x\to 0,
\]
which concludes that the assertion holds for all $x\in\R$.
\end{proof}

\begin{proof}[Proof of Lemma \ref{5:T_diagonal}]
We show that 
$$\I_A\CT^*\I_+\CT\I_B=\I_A\F^*\I_+\F\I_B.$$
The assertion for $\I_-$ follows from conjugating the identity by the inversion operator $\mathcal{J}f(x)=f(-x)$.

Let $h_1, h_2$ be Borel bounded functions supported on $A$ and $B$ respectively. It is sufficient to prove that 
$$
\langle h_1, \CT^*\I_+\CT h_2\rangle_{L_2}=\langle h_1, \F^*\I_+\F h_2\rangle_{L_2}.
$$
By Lemma \ref{4:boundedness} we have that 
\[
\CT^*\I_+\CT = \underset{R\to+\infty}{\wlim}\CT^*\I_{[0, R]}\CT,
\]
which implies
\[
\langle h_1, \CT^*\I_+\CT h_2\rangle_{L_2} = \lim_{R\to\infty}\langle h_1, \CT^*\I_{[0, R]}\CT h_2\rangle_{L_2} = \int_{A\times B}\overline{h_1(x)}h_2(y)\bra*{\int_{0}^R\overline{\CT(xt)}\CT(yt)dt}dxdy.
\]
The identity \eqref{2_eq:CD_cont} yields
\[
\int_{0}^R\overline{\CT(xt)}\CT(yt)dt = R\psi(Rx)\overline{\psi(Ry)}K^s(Rx, Ry).
\]

Recall the following assumption on the Borel sets $A, B$: there exists some $\varepsilon >0$ such that for any $(x, y)\in A\times B$ the inequalities $\abs{x}>\varepsilon$, $\abs{y}>\varepsilon$ and $\abs{x-y}>\varepsilon$ hold. This assumption justifies direct substitution of the asymptotic from Lemma \ref{4:hyper_est} into the expression \eqref{1_eq:CHK_def} for $K^s(Rx, Ry)$. We deduce the formula
\[
\int_{0}^R\overline{\CT(xt)}\CT(yt)dt = \frac{1}{2\pi i(x-y)}\bra*{1-e^{iR(y-x)}\left[\psi(Rx)\overline{\psi(Ry)}\right]^2} + o(1), \quad \text{ as }R\to\infty
\]
where the $o$-term is uniform for $(x, y)\in A\times B$. Observe that the function $\overline{\psi(Ry)}\psi(Rx) = \overline{\psi(x)}\psi(y)$ does not depend on $R$. Consequently, the Riemann-Lebesgue Lemma yields
\begin{multline*}
 \int_{A\times B}\overline{h_1(x)}h_2(y)\bra*{\int_{0}^R\overline{\CT(xt)}\CT(yt)dt}dxdy \\
 =  \int_{A\times B}\overline{h_1(x)}h_2(y)\frac{1}{2\pi i(y-x)}dxdy + o(1), \text{ as } R\to+\infty,
 \end{multline*}
 where the right-hand side equals $\langle h_1, \F^*\I_+\F h_2\rangle_{L_2}$.
\end{proof}

\begin{proof}[Proof of Lemma \ref{5:T_norm}]
By the definition we have
\[
(\CT \I_{[n, n+1]})(x) = \int_n^{n+1}\frac{e^{-ixt}}{\sqrt{2\pi}}Z_s(xt)\rho(xt)\psi(xt)\overline{\psi(xt)^2}dt = \overline{\psi(x)^2}h_n(x),
\]
where
\[
h_n(x) = \int_n^{n+1}\frac{e^{-ixt}}{\sqrt{2\pi}}Z_s(xt)\rho(xt)\psi(xt)t^{2i\Im s}dt.
\]
It is straightforward that $\norm{\CT \I_{[n, n+1]}}_{L_2} = \norm{h_n}_{L_2}$. Introduce
\[
\Delta_n(x) = h_n(x)-\F\left[\I_{[n, n+1]}(t)t^{2i\Im s}\right](x) = \int_n^{n+1}\frac{e^{-ixt}}{\sqrt{2\pi}}\bra*{Z_s(xt)\rho(xt)\psi(xt)-1}t^{2i\Im s}dt
\]
Since $\F$ is unitary, it is sufficient to establish that $\norm{\Delta_n}_{L_2}\to 0$ as $n\to\infty$. Assume first that $\Re s \le 0$. By Lemma \ref{4:hyper_est} we have the estimate
\[
\abs{\Delta_n(x)}\le C\int_n^{n+1}\frac{\abs{xt}^{\Re s}}{1+\abs{xt}^{1+\Re s}}dt = \frac{C}{\abs{x}}\int_{\abs{x}n}^{\abs{x}(n+1)}\frac{t^{\Re s}}{1+t^{1+\Re s}}dt.
\]

For $\abs{x}\ge 1/n$ we may bound the numerator by $1$ to derive that
\[
\abs{\Delta_n(x)}\le \frac{C}{\abs{x}}\int_{\abs{x}n}^{\abs{x}(n+1)}\frac{1}{1+t}dt \le \frac{C}{n\abs{x}}.
\]
Thereby $\norm{\Delta_n\I_{\abs{x}\ge n}}_{L_2}\le C\sqrt{\frac{2}{n}}\to 0$ as $n\to \infty$. For $\abs{x}\le 1/n$ we use the bound
\[
\abs{\Delta_n(x)}\le \frac{C}{\abs{x}}\int_{\abs{x}n}^{\abs{x}(n+1)}t^{\Re s}dt = \frac{C\abs{x}^{\Re s}}{1+\Re s}\bra*{(n+1)^{1+\Re s}-n^{1+\Re s}},
\]
which concludes that as $n\to \infty$
\[
\norm{\Delta_n\I_{\abs{x}\le 1/n}}_{L_2}\le C\sqrt{\frac{2}{(1+2\Re s)n}}\to 0.
\]

In the case $\Re s \ge 0$ we have the bound
\[
\abs{\Delta_n(x)}\le C\int_n^{n+1}\frac{1}{1+\abs{xt}}dt.
\]
The above argument for $\Re s=0$ implies that $\norm{\Delta_n}_{L_2}\to 0$ as $n\to\infty$ in this case as well.
\end{proof}

\section{Proof of Corollary \ref{2:H_decomposition} and Lemma \ref{5:T_szero}}\label{5_sec:H_dec}
\begin{proof}[Proof of Corollary \ref{2:H_decomposition}]
The first statement follows from Theorem \ref{2:Diagonalization} directly.

By Theorem \ref{2:Diagonalization} any function $f\in\mathcal{PW}_s$ may be expressed by the formula
\begin{equation}\label{5_eq:anal_factor}
f(x) = \frac{1}{\sqrt{2\pi}}\rho(x)\int_0^1 e^{ixy}Z_{\bar{s}}(-xy)y^{\bar{s}}g(y)dy = \frac{1}{\sqrt{2\pi}}\rho(x)h_f(x)
\end{equation}
for the function $g(x) = (\CT f)(x)$, which is supported on $[0, 1]$ by the first statement. By the Cauchy-Bunyakovsky-Schwarz inequality the function $\abs{y}^{\bar{s}}g(y)$ is absolutely integrable. The Morera Theorem yields that the function $h_f$ has holomorphic extension to $\C$. This proves the second claim.

To prove the third statement observe that the condition $f\in H^{(s, n)}$ is equivalent to
\[
\int_\gamma z^{-k}h_f(z)dz=0, \text{ for }k=1,\ldots, n
\]
for the contour $\gamma$ encircling the origin. Substituting the expression \eqref{5_eq:anal_factor} for $h_f$ into the formula above, we deduce
\[
\int_\gamma z^{-k}h_f(z)dz = \int_0^1 \bra*{\int_\gamma z^{-k}e^{izy}Z_{\bar{s}}(-zy)dz}y^{\bar{s}}g(y)dy = 2\pi iL_{k-1}\int_0^1y^{k-1}y^{\bar{s}}g(y)dy=0,
\]
where $L_{k-1}$ is the $k-1$-st coefficient of the Taylor expansion of $e^{iz}Z_{\bar{s}}(-z)$ at the origin. We conclude that $f\in H^{(s, n)}$ if and only if the function $g(x) = (\CT f)(x)$ is orthogonal to $y^{j+s}\I_{[0, 1]}(y)$ for $j=0, \ldots, n-1$. This proves the third claim.

Let us show that the spaces $L^{(s, n)}$, which are orthogonal complements of $H^{(s, n+1)}$ in $H^{(s, n)}$, are spanned by
\[
\mathcal{L}_{(s, n)} = \psi^*\CT^*\bra*{\I_{[0, 1]}(t)t^sP_n^{(2\Re s)}(t)},
\]
where $\left\{P_n^{(2\Re s)}\right\}_{n\ge 0}$ are orthogonal polynomials with respect to the weight $t^{2\Re s}$ on $[0, 1]$. Indeed, by the definition of these polynomials they satisfy
\[
\int_0^1 y^{j+\bar{s}}y^sP_n^{(2\Re s)}(y)dy = 0, \text{ for }j=0, \ldots, n-1.
\]
Therefore, the function $\mathcal{L}_{(s, n)}$ belongs to $H^{(s, n)}$. Moreover, the function $\I_{[0, 1]}t^sP_n^{(2\Re s)}(t)$ belongs to the space
\[
F^{(s, n)} = \I_{[0, 1]}(t)t^s\spann\langle 1, t, \ldots, t^n\rangle\subset L_2[0, 1].
\]
By the third claim the image of $H^{(s, n+1)}$ under the operator $\CT\psi$ is orthogonal to the space $F^{(s, n)}$. We conclude that the function $\mathcal{L}_{(s, n)}$ is orthogonal to the space $H^{(s, n+1)}$. This completes the proof of the last statement.
\end{proof}

To prove Lemma \ref{5:T_szero} recall a different characterization of the Hardy space. Define
\[
H^2(\mathbb{H}) =\left \{f\in\mathcal{H}(\mathbb{H}): \sup_{\delta >0}\int_\R\abs{f(x+i\delta)}^2dx<\infty\right\},
\]
where $\mathcal{H}(\mathbb{H})$ stands for the space of functions, which are holomorphic on the upper half-plane $\mathbb{H}$.
\begin{theorem}[The Paley-Wiener Theorem, see {\cite[Theorem~19.2]{R_86}}]\label{5:PW_theorem}
The Hardy space $H^2(\mathbb{H})$ coincides with the space $\F^*L_2(\R_+)$ in the following sense. For any function $f\in H^2(\mathbb{H})$ there exists a function $F\in L_2(\R_+)$ such that
\[
f(x+i\delta) = \int_0^\infty e^{-\omega\delta}F(\omega)e^{i\omega x}d\omega.
\]
\end{theorem}
\begin{proof}[Proof of Lemma \ref{5:T_szero}]

Before proceeding to the proof let us point out the following. Throughout the proof we fix the branch $\arg z\in [0, \pi]$ for the function $p(z)=z^a$, $a\in\mathbb{C}$. We will use that its restriction to $\R\setminus \{0\}$ has analytic extension to the upper half-plane $\mathbb{H}$. We also note that it is uniformly bounded on the closure $\overline{\mathbb{H}}\setminus \{0\}$ if $\Re a = 0$.

By the Paley-Wiener Theorem it is sufficient to establish that $\CT^*\I_{[1/2, 1]}$ has analytic extension to $\mathbb{H}$ and satisfies the growth condition. To prove the first claim observe that by Corollary \ref{2:H_decomposition} we have
\[
\bra*{\CT^*\I_{[1/2, 1]}}(x) = \rho(x)\psi(x)h(x),
\]
where $h$ is an entire function. The identity
\[
\rho(x)\psi(x) = \abs{x}^{\bar{s}}e^{-\frac{i\pi}{2}\bar{s}\sign x} = x^{\bar{s}}e^{-\frac{i\pi}{2}\bar{s}} \text{ for }x\in \R\setminus \{0\}
\]
implies that $\rho(x)\psi(x)$ has analytic extension to the upper half-plane $\mathbb{H}$, which is equal to $z^{\bar{s}}e^{-\frac{i\pi}{2}\bar{s}}$ for $z\in\mathbb{H}$. Therefore the function $\bra*{\CT^*\I_{[1/2, 1]}}(x)$ also extends to $\mathbb{H}$.

To check the growth condition define the function
\[
\tilde{\CT}(z) = \frac{e^{iz}}{\sqrt{2\pi}}z^{\bar{s}}e^{-\frac{i\pi}{2}\bar{s}}Z_{\bar{s}}(-z), \quad z\in\overline{\mathbb{H}}\setminus \{0\},
\]
which is holomorphic on $\mathbb{H}$. The expansion \eqref{A_eq:asymp} yields that for a certain constant $C$ the bound
\begin{equation}\label{5_eq:H_est}
\abs*{z^{\bar{s}}e^{-\frac{i\pi}{2}\bar{s}}Z_{\bar{s}}(-z) - e^{-\pi\Im s}z^{-2i\Im s}}\le \begin{cases}
\frac{C\abs{z}^{\Re s}}{1+\abs{z}^{1+\Re s}}, \text{ if }\Re s<0,\\
\frac{C}{1+\abs{z}}, \text{ otherwise}
\end{cases}
\end{equation}
holds for all $z\in\overline{\mathbb{H}}\setminus \{0\}$.

Next define the functions
\[
g(z) = z^{-2i\Im s}u(z), \quad u(z)= e^{-\pi\Im s}\int_{1/2}^1e^{itz}t^{-2i\Im s}dt, \quad z\in \overline{\mathbb{H}}\setminus \{0\}.
\]
The function $g$ belongs to the Hardy space $H^2(\mathbb{H})$. Indeed, by the definition the Fourier transform of the restriction $u|_{\mathbb{R}}$ is supported on $\mathbb{R}_+$. Hence $u$ belongs to the Hardy space $H^2(\mathbb{H})$ by the Paley-Wiener Theorem. Since $z^{i\Im s}$ is holomorphic and bounded on $\mathbb{H}$, the function $g$ belongs to the Hardy space $H^2(\mathbb{H})$ by the definition of the latter.

We have that the analytic extension of $\bra*{\CT^*\I_{[1/2, 1]}}(x)$ to $\mathbb{H}$ is equal to
\[
\bra*{\CT^*\I_{[1/2, 1]}}(z) = \int_{1/2}^1\tilde{\mathcal{T}}(zt)dt, \quad z\in\mathbb{H}.
\]
Since the defined function $g(z)$ belongs to the Hardy space, it is remaining to show that
\[
(\CT^*\I_{[1/2, 1]})(z)-g(z)\in H^2(\mathbb{H}).
\]
This implies that $\bra*{\CT^*\I_{[1/2, 1]}}(x)\in \F^*L_2(\R_+)$ by the Paley-Wiener Theorem.

Let us check the growth condition for the difference. Assume first that $\Re s<0$. By the estimate \eqref{5_eq:H_est} we have
\[
\norm*{\CT^*\I_{[1/2]}(\cdot+i\delta)-g(\cdot+i\delta)}_{L_2}^2\le C^2\int_{1/2}^1\int_\R\frac{\abs{(x+i\delta)t}^{2\Re s}}{(1+\abs{(x+i\delta)t}^{1+\Re s})^2}dxdt.
\]
Consider the cases $\delta\ge 1$ and $\delta<1$ separately
\[
\norm*{\CT^*\I_{[1/2]}(\cdot+i\delta)-g(\cdot+i\delta)}_{L_2}^2 \le 
\begin{cases}
\begin{aligned}
&C^2\int_{1/2}^1\frac{1}{t^2}dt\int_\R\frac{1}{1+x^2}dx = C_{\delta\ge 1}, \text{ if }\delta\ge 1,\\
&\frac{C^2}{2}\int_\R\frac{A+\abs{x}^{2\Re s}}{(1+\frac{1}{2^{1+\Re s}}\abs{x}^{1+\Re s})^2}dx = C_{\delta\le 1}, \text{ if }\delta <1.
\end{aligned}
\end{cases}
\]
We conclude that the growth condition holds
\[
\sup_{\delta > 0} \int_\R\abs*{\CT^*\I_{[1/2]}(x+i\delta)-g(x+i\delta)}^2dx \le C_{\delta\ge 1}+C_{\delta\le 1} <+\infty,
\]
and hence the function $\CT^*\I_{[1/2, 1]}-g$ belongs to the Hardy space. 

In the case $\Re s >0$ the bound \eqref{5_eq:H_est} implies the inequality
\[
\norm*{\CT^*\I_{[1/2]}(\cdot+i\delta)-g(\cdot+i\delta)}_{L_2}^2\le C^2\int_{1/2}^1\frac{1}{t}dt\int_\R \frac{1}{(1+\abs{x})^2}dx,
\]
which similarly proves the desired claim.

Lemma \ref{5:T_szero} is proved.
\end{proof}

\section{Wiener-Hopf factorization of $G_f$}\label{6_sec:WH_fact}
    We refer to \cite{S_05} for the introduction to trace-class and Hilbert-Schmidt operators. Recall that we denote bounded Borel functions and the respective operators of pointwise multiplication by one symbol.
    
    Let us show how Corollary \ref{2:WH_factorization} follows from Theorem \ref{2:PW_spaces}. The first assertion for Wiener-Hopf operators follows from the property of the convolution $\supp f*g\subset \supp f+\supp g$. Denote $\mathcal{U}=\CT\F^*$. By Theorem \ref{2:PW_spaces} we have
    \[
    \F^*W_f\F = \F^*\I_+\F f\F^*\I_+\F = \CT^*\I_+\CT f\CT^*\I_+\CT = \CT^*G_f\CT,
    \]
    which yields
    \[
    G_f = \mathcal{U}W_f\mathcal{U}^*.
    \]
    For $f_\pm \in\F^*L_1(\R_\pm)$ we conclude
    \[
    G_{f_+f_-} = \mathcal{U}W_{f_+f_-}\mathcal{U}^* = \mathcal{U}W_{f_-}\mathcal{U}^*\mathcal{U}W_{f_+}\mathcal{U}^*=G_{f_-}G_{f_+}.
    \]
    The rest of the relations follow similarly.
    
    To prove the second assertion we first recall the derivation for the Wiener-Hopf operators. Substitution of the first claim $W_{f_+f_-} = W_{f_-}W_{f_+}$ into the formula for the commutator gives
    \begin{multline*}
    [W_{f_-}, W_{f_+}] = W_{f_+f_-}-W_{f_+}W_{f_-} = \I_+\F f_+\F^*\F f_-\F^*\I_+-\I_+\F f_+\F^*\I_+\F f_-\F^*\I_+ =\\
    = \I_+\F f_+\F^*\I_-\F f_- \F^*\I_+.
    \end{multline*}
    The claim of the commutator being trace-class follows from the operators $\I_+\F f_+\F^*\I_-$, $\I_-\F f_- \F^*\I_+$ each being Hilbert-Schmidt. Indeed, these are integral operators with the kernels $K_\pm(x, y) = \hat{f}_\pm(x-y)$. A direct calculation of $L_2$ norms of the kernels gives
    \[
    \int_{\R_+}dx\int_{\R_-}dy \abs{\hat{f}_+(x-y)}^2 = \int_0^\infty y\abs{\hat{f}(y)}^2dy,
    \]
    which is finite since $f\in H_{1/2}(\R)$. The argument for $f_-$ is similar.
    
    To calculate the trace we use Mercer's theorem (see \cite[Theorem~3.11.9]{S_15}). Mercer's theorem states that for an integral trace-class operator with a continuous kernel $K(x, y)$ its trace is equal to the integral of $K(x, x)$. In our case
    \[
    \Tr [W_{f_-}, W_{f_+}] = \int_{\R_+}dx\int_{\R_-}dy \hat{f}_+(x-y)\hat{f}_-(y-x) = \int_0^\infty y \hat{f}(y)\hat{f}(-y)dy.
    \]
    
    The assertion for general $s$ follows from
    \[
    [G_{f_-}, G_{f_+}] = \mathcal{U}[W_{f_-}, W_{f_+}]\mathcal{U}^*.
    \]
    Corollary \ref{2:WH_factorization} is proven completely.
    \appendix

\section{Hypergeometric functions}\label{appendix:hg_fun}

Recall that hypergeometric functions \cite[13.1.2, 15.1.1]{AS_64} are defined by the formulae
\begin{equation}\label{A_eq:def}
  \FO{a}{b}{z} = \sum_{k=0}^\infty \frac{(a)_k}{(b)_kk!}z^k,\qquad \FT{a}{b}{c}{z} = \sum_{k=0}^\infty \frac{(a)_k(b)_k}{(c)_kk!}z^k,
\end{equation}
where $(a)_k = a(a+1)\ldots (a+k-1)$. Hypergeometric functions have the integral representations (see \cite[13.2.1, 15.3.1]{AS_64}):
\begin{align}\label{A_eq:int}
  &\FO{a}{b}{z} = \frac{\Gamma(b)}{\Gamma (a)\Gamma(b-a)}\int_0^1 t^{a-1}(1-t)^{b - a - 1}e^{tz}dt,\\
  \label{A_eq:int_2}
  &\FT{c}{a}{b}{z} = \frac{\Gamma(b)}{\Gamma(a)\Gamma(b-a)}\int_0^1t^{a-1}(1-t)^{b-a-1}(1-tz)^{-c}dt.
\end{align}
The asymptotic \cite[13.5.1]{AS_64} of the function ${}_1F_1$ as $z\to\infty$ is
\begin{equation}\label{A_eq:asymp}
\FO{a}{b}{z}=\frac{\Gamma(b)}{\Gamma(b-a)}e^{\pm i\pi a}z^{-a}\bra*{1+O(\abs{z}^{-1})}+\frac{\Gamma(b)}{\Gamma(a)}e^zz^{a-b}\bra*{1+O(\abs{z}^{-1})},
\end{equation}
where the upper sign being taken if $\arg z\in (-\pi/2, 3\pi/2)$, the lower sign if $\arg z\in (-3\pi/2, -\pi/2]$.

Last, we mention Kummer's formula \cite[13.1.27]{AS_64}
\begin{equation}\label{A_eq:kummer}
  \FO{a}{b}{z} = e^z\FO{b-a}{b}{-z}.
\end{equation}

Let us proceed to their application to orthogonal polynomial ensembles. For $\Re s > -1/2$ introduce
\[
    w_s(z) = \frac{1}{2\pi}\Gam{1+s, 1+\bar{s}}{1+2\Re s}(1-z)^{\bar{s}}(1-\bar z)^s, \quad \abs{z}=1.
\]
Orthogonal polynomials with respect to the weight $w_s(z)d\theta$, $z=e^{i\theta}$ may be expressed in terms of hypergeometric functions.

See \cite[p. 369]{BNR_06}, \cite[p. 31-34]{BC_08} for the theorem below. Basor and Chen note in \cite[p. 34]{BC_08} that the orthogonality of Hypergeometric functions was found by Askey in commentary on Szeg\H{o}'s collected papers \cite[p. 304]{AS_82}.
\begin{theorem}\label{A:orth_formula}
Monic orthogonal polynomials $\{\Phi_n\}_{n\ge 0}$ with the weight $w_s(z)d\theta$ have the following expression
\[	
\Phi_n(z) = \Gam{s+\bar{s} + 1 + n, \bar{s}+1}{\bar{s}+n+1, s+\bar{s}+1}\FT{-n}{\bar{s}+1}{s+\bar{s}+1}{1-z}.
\]
Further, their norm is equal to
\[
\norm{\Phi_n}^2_{L_2(\T, w_s(\theta)d\theta)} = \Gam{s+\bar{s}+1+n, n+1, s+1, \bar{s}+1}{\bar{s}+n+1, s+n+1, s+\bar{s}+1}.
\]
\end{theorem}

\end{document}